\documentclass[a4paper]{amsart}
\usepackage[latin1]{inputenc}
\usepackage{amssymb}
\usepackage{amsmath}
\usepackage{mathrsfs}
\usepackage{amsthm}
\usepackage{mathabx}
\usepackage{amsfonts}
\usepackage{textcomp}
\usepackage{graphicx}
\usepackage[pdftex]{color}
\usepackage{paralist}
\usepackage[shortlabels]{enumitem}
\usepackage{hyperref}
\usepackage{comment}
\usepackage[arrow, matrix, curve]{xy}
\usepackage{enumitem}

\usepackage{tikz-cd}

\newtheorem*{corollary*}{Corollary}
\newtheorem*{theorem*}{Theorem}
\newtheorem{theorem}{Theorem}[section]

\newtheorem{corollary}[theorem]{Corollary}
\newtheorem{lemma}[theorem]{Lemma}
\newtheorem{proposition}[theorem]{Proposition}

\newtheorem*{claim*}{Claim}

\theoremstyle{definition}
\newtheorem{definition}[theorem]{Definition}

\newtheorem*{theorem }{Theorem}
\newtheorem{remark}[theorem]{Remark}
\newtheorem{example}[theorem]{Example}

\theoremstyle{remark}

\numberwithin{equation}{theorem}

\makeatletter
\renewcommand*\env@matrix[1][\
arraystretch]{%
  \edef\arraystretch{#1}%
  \hskip -\arraycolsep
  \let\@ifnextchar\new@ifnextchar
  \array{*\c@MaxMatrixCols c}}
\makeatother

\renewcommand{\mod}{\operatorname{mod-}\!\!}
\newcommand{\lmod}{\!\operatorname{-mod}}

\newcommand{\Ext}{\operatorname{Ext}}
\newcommand{\End}{\operatorname{End}}
\newcommand{\pd}{\operatorname{pd}}
\newcommand{\id}{\operatorname{id}}
\newcommand{\findim}{\operatorname{findim}}

\newcommand{\Hom}{\operatorname{Hom}}
\newcommand{\add}{\operatorname{\mathrm{add}}}
\renewcommand{\top}{\operatorname{\mathrm{top}}}

\renewcommand{\ker}{\mathrm{Ker}}

\newcommand{\domdim}{\operatorname{domdim}}

\newcommand{\Fac}{\operatorname{Fac}}

\newcommand{\F}{\mathcal{F}}
\newcommand{\St}{\Delta}
\newcommand{\pSt}{\overline{\Delta}}
\newcommand{\Cs}{\nabla}
\newcommand{\pCs}{\overline{\nabla}}
\begin{document}

\title{On properly stratified Gorenstein algebras}
\date{\today}


\author{Tiago Cruz}
\address{Institute of algebra and number theory, University of Stuttgart, Pfaffenwaldring 57, 70569 Stuttgart, Germany}
\email{tiago.cruz@mathematik.uni-stuttgart.de}
\author{Ren\'{e} Marczinzik}
\address{Institute of algebra and number theory, University of Stuttgart, Pfaffenwaldring 57, 70569 Stuttgart, Germany}
\email{marczire@mathematik.uni-stuttgart.de}

\begin{abstract}
We show that a properly stratified algebra is Gorenstein if and only if the characteristic tilting module coincides with the characteristic cotilting module. We further show that properly stratified Gorenstein algebras $A$ enjoy strong homological properties such as all Gorenstein projective modules being properly stratified and all endomorphism rings $\End_A(\Delta(i))$ being Frobenius algebras.
We apply our results to the study of properly stratified algebras that are minimal Auslander-Gorenstein algebras in the sense of Iyama-Solberg and calculate under suitable conditions
their Ringel duals. This applies in particular to all centraliser algebras of nilpotent matrices.
\end{abstract}

\maketitle
\section*{Introduction}
Quasi-hereditary algebras constitute an important class of finite dimensional algebras including many well-studied algebras such as algebras of global dimension at most two,  Schur algebras, see for example \cite{D} or \cite{Gre}, and blocks of category $\mathcal{O}$, see for example \cite{H}.

Standardly stratified algebras were introduced as a generalisation of quasi-hereditary algebras in \cite{CPS}. It can be shown that a standardly stratified algebra $A$ is quasi-hereditary if and only if $A$ has finite global dimension and standardly stratified algebras include many important algebras with infinite global dimension arising for example in Lie theory, see \cite{M2}. 
Standardly stratified algebras always have a characteristic tilting module but in general no characteristic cotilting module. When the opposite algebra of a standardly stratified algebra is also standardly stratified, the algebra is properly stratified.
The properly stratified algebras where the characteristic tilting and cotilting modules coincide are of particular importance as they have many strong homological properties that are in general missing in the more general case of standardly stratified algebras, see for example \cite{FrMa}. 
In \cite{MP}, Mazorchuk and Parker conjectured that the finitistic dimension of any properly stratified algebra with a simple preserving duality is equal to two times the projective dimension of the characteristic tilting module.
This conjecture was proven by Mazorchuk and Ovsienko in \cite{MazOv} under the additional assumption that the characteristic tilting module coincides with the characteristic cotilting module and in \cite{M} it was shown that the conjecture is wrong without this assumption.
Our main result gives a new homological characterisation when the characteristic tilting module coincides with the characteristic cotilting module and implies for example that the main theorem of Mazorchuk and Ovsienko can be simplified as the finitistic dimension coincides with the Gorenstein dimension of the algebra under their assumptions.
\begin{theorem*}[\ref{mainresult}]
For a properly stratified algebra $A$, the characteristic tilting module coincides with the characteristic cotilting module if and only if $A$ is Gorenstein.
\end{theorem*}

We show that properly stratified Gorenstein algebras enjoy strong homological properties. The next theorem collects the most important properties.
\begin{theorem*}[\ref{projectiveisproperlystratified}, \ref{endofstandardisfrobenius}]
Let $A$ be a properly stratified Gorenstein algebra with standard modules $\Delta$.
\begin{enumerate}
\item Every Gorenstein projective $A$-module is properly stratified (i.e. it belongs to $\mathcal{F}(\overline{\Delta})$).
\item A module $M$ is in $F(\Delta)$ if and only if $M$ is in $F(\overline{\Delta})$ with finite projective dimension.
\item All endomorphism rings $\End_A(\Delta(i))$ are Frobenius algebras.

\end{enumerate}

\end{theorem*}

We introduce GIGS algebras as gendo-symmetric properly stratified Gorenstein algebras having a duality. Those algebras generalise the algebras introduced by Fang and Koenig in \cite{FanKoe3} from the quasi-hereditary scenario to the more general case of properly stratified algebras. We will see that central results obtained by Fang and Koenig still hold even when the algebras are not necessarily quasi-hereditary. The class of GIGS algebras contain all Schur algebras $S(n,k)$ for $n \geq k$, blocks of category $\mathcal{O}$ and we will show that they also contain centraliser algebras of nilpotent matrices, which are properly stratified but in general not quasi-hereditary.  
By the result of Mazorchuk-Ovsienko, GIGS algebras always have even Gorenstein dimension and since they are gendo-symmetric by definition they are always isomorphic to an algebra of the form $\End_U(X)$ for a symmetric algebra $U$ and a generator $X$ of $\mod U$. One of the most important subclasses of Gorenstein algebras are the minimal Auslander-Gorenstein algebras introduced by Iyama and Solberg in \cite{IyaSol} as a generalisation of higher Auslander algebras that were introduced in \cite{Iya}.
Our main result for GIGS algebras that are minimal Auslander-Gorenstein is as follows:

\begin{theorem*}[\ref{Ringeldualtheorem}]
Let $A=\End_U(U \oplus M)$ be a GIGS algebra with a symmetric algebra $U$ and a generator $U \oplus M$ of $\mod U$, where we can assume that $M$ has no projective direct summands.
Assume $A$ is furthermore a minimal Auslander-Gorenstein algebra with Gorenstein dimension equal to $2d$ for some $d \geq 1$.
Then the Ringel dual $R_A$ of $A$ is isomorphic to $\End_U(U \oplus \Omega^d(M))$. In particular, $R_A$ is again a GIGS algebra that is minimal Auslander-Gorenstein with Gorenstein dimension $2d$.

\end{theorem*}
Recall that for an $n \times n$-matrix $M$ with entries in a field $K$, the centralizer algebra of $M$ is defined as the $K$-algebra of all $n \times n$-matrices $X$ with $XM=MX$.
In Subsection \ref{Centraliser algebras of nilpotent matrices}, we will apply the previous theorem to calculate the Ringel duals of centraliser algebras of nilpotent matrices and determine when they are Ringel self-dual.

\section{Preliminaries}
In this paper, unless stated otherwise, all algebras under consideration are finite dimensional over a field $K$ and all modules are right modules.
We assume the reader is familiar with the basic representation theory and homological algebra of finite dimensional algebras and refer for example to \cite{SkoYam} and \cite{ARS} for a basic introduction.
For a subcategory $\mathcal{C}$ of $\mod A$ for a finite dimensional algebra $A$, we denote by $\add(\mathcal{C})$ the full subcategory of direct sums of direct summands of a module $X \in \mathcal{C}$. For a module $X$, $\add(X)$ simply denotes the full subcategory of modules that are direct summands of $X^n$ for some $n$. By $D$ we denote the functor $\Hom_K(-, K)$ and $J$ denotes the Jacobson radical of an algebra $A$.
For a subcategory $\mathcal{U}$ of $\mod A$ we denote by $\widehat{\mathcal{U}}$ (or $\widecheck{\mathcal{U}}$) the full subcategory of $\mod A$ consisting of modules $X$ such that there is an exact sequence $0 \rightarrow C_n \rightarrow \cdots \rightarrow C_0 \rightarrow X \rightarrow 0$ (or $0 \rightarrow X \rightarrow C_0 \rightarrow \cdots \rightarrow C_n \rightarrow 0$) with $C_i \in \mathcal{U}$. 
For a module $X$ we define ${}^{\perp}X:= \{ Y \in \mod A | \Ext_A^i(Y,X)=0$ for $i>0 \} $ and $X^{\perp}:= \{ Y \in \mod A | \Ext_A^i(X,Y) =0 $ for $i >0 \}$.
By $\Fac(X)$ we denote the full subcategory of modules that are an epimorphic image of $X$.
We denote by $P^{<\infty}$ the full subcategory of modules having finite projective dimension and by $I^{<\infty}$ the full subcategory of modules having finite injective dimension. 
A subcategory $\mathcal{U}$ of $\mod A$ is called \emph{resolving} if it contains all projective $A$-modules and is closed under direct summands, extensions and kernels of epimorphisms.
A \emph{coresolving} subcategory is defined dually.
\subsection{Gorenstein algebras and cotilting modules}

An algebra $A$ is called \emph{Gorenstein} in case the left and right injective dimensions of the regular module $A$ are finite, in which case they coincide and the common number is the \emph{Gorenstein dimension} of $A$. 
A \emph{selfinjective algebra} is a Gorenstein algebra with Gorenstein dimension zero, or equivalently an algebra where all projective modules are injective. $A$ is called \emph{Frobenius algebra} when $D(A) \cong A$ as $A$-right modules and it is called \emph{symmetric} when $D(A) \cong A$ as $A$-bimodules. Every Frobenius algebra is selfinjective and every selfinjective quiver algebra is a Frobenius algebra.
An algebra $A$ is called \emph{gendo-symmetric} if $D(A) \otimes_A D(A) \cong D(A)$ or equivalently when $A$ is isomorphic to $\End_B(M)$ for a symmetric algebra $B$ with a generator $M$ of $\mod B$. We refer to \cite{FanKoe} and \cite{FanKoe2} for more information and properties of gendo-symmetric algebras.
The \emph{finitistic dimension} of an algebra $A$ is defined as the supremum of all projective dimensions of modules having finite projective dimension. A module $C\in \mod A$ is called \emph{cotilting} if it has no self-extensions, finite injective dimension and $DA\in \widehat{\add C}$. Tilting modules are defined, analogously.

\begin{theorem} \label{gorenstein characterisation}
The following are equivalent:
\begin{enumerate}
\item $A$ is Gorenstein.
\item A module is a tilting module if and only if it is a cotilting module.
\item $P^{< \infty}=I^{< \infty}$.
\end{enumerate}
\end{theorem}
\begin{proof}
See \cite[Lemma 1.3]{HU}.
\end{proof}

\begin{proposition} \label{auslander-reitenfindimresult}
	Let $A$ be an algebra such that $A$ and $A^{op}$ have finite finitistic dimension.
	In case the left or right injective dimension of $A$ is finite, $A$ is Gorenstein. 
	In this case the Gorenstein dimension coincides with the finitistic dimension of $A$ and $A^{op}$.
\end{proposition}
\begin{proof}
	See \cite[Proposition 6.10]{AR}.
	That the finitistic dimension coincides with the Gorenstein dimension for Gorenstein algebras can be found for example in \cite[Lemma 2.3.2]{Che}.
\end{proof}

Recall that a pair $(\mathcal{X},\mathcal{Y})$ of subcategories of $\mod A$ is called a \emph{cotorsion pair} when $\mathcal{Y}= \{ M \in \mod A | \Ext_A^1(\mathcal{X},M)=0 \}$ and $\mathcal{X}= \{ N \in \mod A | \Ext_A^1(N,\mathcal{Y})=0 \}$. A cotorsion pair is called \emph{complete} when $\mathcal{X}$ is contravariantly finite. 
The next result gives a correspondence between basic cotilting modules and certain complete cotorsion pairs. 
\begin{theorem} \label{cotiltcorres} 
	Let $A$ be a general finite dimensional algebra.
	There is a one-one correspondence between basic cotilting modules $U$ and complete cotorsion pairs $(\mathcal{X},\mathcal{Y})$ with $\mathcal{X}$ resolving and $\widehat{\mathcal{X}}=\mod A$, given by $U \rightarrow ({}^{\perp}U, \widehat{\add(U)})$ and $(\mathcal{X},\mathcal{Y})$ $\rightarrow $ the direct sum of all indecomposable modules in $\mathcal{X} \cap \mathcal{Y}$.
	
\end{theorem}
\begin{proof}
	See \cite[2.3 (b)]{Rei}.
\end{proof}

\subsection{Standardly and properly stratified algebras}

For the convenience of the reader, we will recall some definitions and elementary properties involving stratified algebras.

Let $\textbf{e}=\{e_1, \ldots, e_n\}$ be a linearly ordered complete set of primitive orthogonal idempotents. For each $i=1, \ldots, n$, we write \begin{align*}
	P(i)=e_iA, \ S(i)=\top (e_iA), \ I(i)=D(Ae_i), \ \St(i)=P(i)/e_i A(e_{i+1}+\cdots+ e_n)A.
\end{align*}
That is, $\St(i)$ is the maximal quotient of $P(i)$ without composition factors $S(j)$, $j>i$. Define $\pSt(i)$ to be the maximal quotient of $\St(i)$ having only once $S(i)$ as composition factor.
We call $\St=\{\St(1), \ldots, \St(n)\}$ the (right) standard modules of $A$ and $\pSt=\{\pSt(1), \ldots, \pSt(n) \}$ the proper standard modules. Let $\St_{op}$ and $\pSt_{op}$ be the standard and proper standard modules of $A^{op}$, respectively. By the costandard and proper costandard (right) modules of $A$ we mean $\Cs(i)=D\St_{op}(i)$ and $\pCs(i)=D\pSt_{op}(i)$.

For a given set of modules $\Theta$, by $\F(\Theta)$ we mean the full subcategory of $\mod A$ having a filtration by the modules in the set $\Theta$.  

\begin{definition}
	An algebra together with a linearly ordered complete set of orthogonal primitive idempotents $(A, \textbf{e})$ is called 
	\begin{itemize}
		\item \emph{standardly stratified} if $A_A\in \F(\St)$;
		\item \emph{properly stratified} if $A_A\in \F(\St)\cap \F(\pSt)$.
	\end{itemize} 
\end{definition}

Note that $A_A\in \F(\pSt)$ if and only if $DA_A\in \F(\pCs_{op})$. This last condition is equivalent to $A^{op}$ being standardly stratified (see \cite[Theorem 1.1]{ADLU}, \cite[2.2]{ADL} and \cite{Lak}).

The next theorem collects some results on properly stratified algebras:
\begin{theorem} \label{propstratresults}
Let $(A, \textbf{e})$ be a properly stratified algebra. Then the following assertions hold:
\begin{enumerate}[label=\upshape(\alph*),ref=\thetheorem (\alph*)]
	\item $\F(\St)=\{M\in \mod A | \Ext_A^1(M, \F(\pCs))=0 \}$ is a resolving subcategory of $\mod A$.
	\item $\F(\pCs)= \{N\in \mod A | \Ext_A^1(\F(\St), N)=0 \}$ is a coresolving subcategory of $\mod A$.
\item $\F(\Delta) \subseteq P^{< \infty}$ and $\F(\nabla) \subseteq I^{< \infty}$. \label{propstratresults1}
\item $\Ext_A^k(\F(\St),\F(\pCs))=0$ for all $k \geq 1$. \label{propstratresults2}
\item There exists a (unique) tilting module $T$ such that $\F(\pCs)=T^{\perp}$ and $\widecheck{\F(\pCs)}=\mod A$. \label{propstratresultse}
\item There exists a unique basic tilting module $T$ such that $\add T=\F(\St)\cap \F(\pCs)$.\label{propstratresultsf}
\item There exists a (unique) cotilting module $C$ such that $\F(\pSt)={}^{\perp}C$ and $\widehat{\F(\pSt)}=\mod A$. \label{propstratresultsg}
\item There exists a unique basic tilting module $C$ such that $\add C=\F(\pSt)\cap \F(\Cs)$. \label{propstratresultsh}
\item $T^{\perp} \subseteq Fac(T)$. \label{propstratresults3}
\end{enumerate}
\end{theorem}
\begin{proof}
Assertions (a) and (b) 	are the content of Theorem 1.6 of \cite{ADLU}. For (c), see \cite[Proposition 1.3]{PR}. Assertion (d) follows from $A^{op}$ being standardly stratified and from \cite[Theorem 3.1]{ADL}.
For assertions (e) and (f) see Theorem 2.1 and Proposition 2.2 of \cite{ADLU}. See also Theorem 3.3 of \cite{Xi}.
Assertions (g) and (h) follow from assertions (e) and (f) since $A^{op}$ is standardly stratified.
 For (i), see \cite{ADLU} on page 151.
\end{proof}

We call the tilting module $T$ in the conditions of Theorem $\ref{propstratresultse}$ and $\ref{propstratresultsf}$ the \emph{characteristic tilting module} of the properly stratified algebra  $(A, \textbf{e})$. Dually, we call the cotilting module $C$ in the conditions $\ref{propstratresultsg}$ and $\ref{propstratresultsh}$ the \emph{characteristic cotilting module} of the properly stratified algebra  $(A, \textbf{e})$.
By the previous theorem we have that $\F(\St)= \widecheck{\add T}$, $\F( \pCs)=T^{\perp}$ and $(\F(\St), \F( \pCs))$ is a complete cotorsion pair, see also \cite[Theorem 3.6]{Rei} for this and the dual result.

The following lemma although being elementary and in some sense folklore in the literature of stratified algebras will be useful afterwards to characterize Gorenstein properly stratified algebras.

\begin{lemma}\label{filtrations of standard in proper standard}
	Let $(A, \textbf{e})$ be a standardly stratified algebra. Then the following assertions hold.
	\begin{enumerate}[(1)]
		\item $Ae_nA$ is projective as right $A$-module and $(A/Ae_nA, \{\underline{e_1}, \ldots, \underline{e_{n-1}}\})$ is standardly stratified, where $\underline{e}$ is the image of $e$ in the quotient $A/Ae_nA$;
		\item $Ae_n$ is a projective generator of $e_nAe_n$;
	\end{enumerate}
Moreover, if $(A, \textbf{e})$ is properly stratified then  $\St(i)\in \F(\pSt(i))$, for each $i=1, \ldots, n$.
\end{lemma}
\begin{proof}
	Taking into account that each $\St(j)$ with $j\neq n$ belongs to $\mod A/Ae_nA$ and filtrations of $A$ can be chosen so that higher indexes appear at the bottom, $A$ having a filtration by standard modules gives that $A/Ae_nA\in \F(\St(i)_{i\neq n})$ (see for example \cite[Lemma A2.2]{DK}).
	
	Since $\F(\St)$ is a resolving subcategory of $\mod A$, the exact sequence
	\begin{align*}
		0\rightarrow Ae_nA\rightarrow A\rightarrow A/Ae_nA\rightarrow 0
	\end{align*}yields that $Ae_nA\in \F(\St)$. Since $Ae_nA\in Fac(e_nA)$ and all $\St(j)$ have no composition factors of the form $S(n)$ we must have $Ae_nA\in \F(\St(n))$. It follows that $Ae_nA$ is projective and $\add Ae_nA=\add e_nA$. In particular, $\Hom_A(e_nA, A/Ae_nA)=0$. This shows (1).

Write $e=e_n$. Since $AeA$ is projective we can write $Ae\otimes_{eAe}eA\cong AeA$. Therefore,
\begin{align*}
\add_{eAe}	Ae=\add_{eAe} Ae\otimes_{eAe}eAe=\add_{eAe} AeAe=\add_{eAe} \Hom_A(eA, AeA)= \add_{eAe} \End_A(eA).
\end{align*}This shows (2).

Assume now that $(A, \textbf{e})$ is properly stratified. Then $(A^{op}, \textbf{e})$ is standardly stratified and $\F(\pCs_{op})$ is a coresolving subcategory of $\mod A^{op} $ (see \cite[Theorem 1.6]{ADLU}).
Thus, $\F(\pSt)$ is a resolving subcategory of $\mod A$. Therefore, $\St(n)\in \F(\pSt)$. Further, since $\Ext^1(\pSt(i), \pSt(j))=0$, for $i>j$ (see \cite[Lemma 1.2]{ADLU} ) we can rearrange, if necessary, the filtration of modules in $\F(\pSt)$ so that proper standard modules with lower index appear in the top of the filtration. So, there exists a surjective map $\St(n)\rightarrow \pSt(j)$, where $j$ is the lowest index in the filtration of $\St(n)$. Then $S(n)=\top \St(n)\rightarrow \top \pSt(j)=S(j)$ is surjective. So $j$ must be $n$. Hence, $\St(n)\in \F(\pSt(n))$. Since each $\pSt(j)$ has no composition factors of the form $S(i)$ with $i>j$ the last claim follows by induction.
\end{proof}

\begin{proposition} \label{extlemmaidemideal}
	Let $A$ be a finite dimensional algebra with an idempotent $e$ such that the ideal $AeA$ is projective as a right $A$-module.
	Then $\Ext_{A/AeA}^i(M,N)=\Ext_A^i(M,N)$ for all $M,N \in \mod A/AeA$.
\end{proposition}
\begin{proof}
	See for example \cite[Example 1]{APT}.
\end{proof}

\begin{theorem} \label{finitisticbound}
Let $(A, \textbf{e})$ be a standardly stratified algebra. Then $A$ has finite finitistic dimension.
\end{theorem}
\begin{proof}
See \cite[Corollary 2.7]{ADLU}.
\end{proof}

Let $A$ be a properly stratified algebra with the associated subcategory $\mathcal{F}(\Delta^A)$ and $B$ a properly stratified algebra with the associated subcategory $\mathcal{F}(\Delta^B)$. Then $A$ and $B$ are said to be \emph{equivalent} as properly stratified algebras when $\mathcal{F}(\Delta^A)$ and $\mathcal{F}(\Delta^B)$ are exact equivalent. When $T$ is the characteristic tilting module of $A$ and $R_A:=\End_A(T)$ the \emph{Ringel dual} of $A$ (which is also properly stratified by \cite[Theorem 3.7]{Rei}), then $A$ is said to be \emph{Ringel self-dual} if $A$ and $R_A$ are equivalent as properly stratified algebras.
This generalises the classical notion of being Ringel self-dual for quasi-hereditary algebras, see for example \cite[Seciton 5]{CE}.

\subsection{Standardly stratified algebras with a duality}

The algebra $(A, \textbf{e})$ is said to have a \emph{duality} $\iota$ if $\iota$ is an anti-automorphism of $A$ fixing the complete set of primitive orthogonal idempotents $\textbf{e}$ and $\iota^2=\id_A$. 
Algebras with a duality were studied in \cite{FanKoe3}. The existence of the duality implies the existence of exact functors $(-)^\iota\colon \mod A\rightarrow A\lmod$ and ${}^\iota(-)\colon A\lmod\rightarrow \mod A$ satisfying, in particular, $(e_iA)^\iota\cong Ae_i$ and ${}^\iota(M^\iota)\cong M$, for every $M\in \mod A$. Denote by $(-)^\natural\colon \mod A\rightarrow \mod A$ the composition of functors $D\circ (-)^\iota$. In addition, assume in the remaining of this subsection that $(A, \textbf{e})$ is a standardly stratified algebra. Then $P(i)^\natural\cong D(Ae_i)=I(i)$ and $S(i)^\natural\cong S(i)$. So, $(-)^\natural$ is  a simple preserving duality of $A$. Moreover,  $\St(i)^\natural\cong \Cs(i)$ and $\pSt(i)^\natural\cong \pCs(i)$, for all $i$. Applying $D$ to the isomorphism $\St(i)^\natural\cong \Cs(i)$ we obtain $\St(i)^\iota\cong D\Cs(i)\cong \St_{op}(i)$. It follows that ${}_A A\cong (A_A)^\iota\in \F(\St^\iota)=\F(\St_{op})$. Therefore, a standardly stratified algebra $(A, \textbf{e})$ with a duality is a properly stratified algebra (see also \cite[Proposition 2.3]{CheDl}).

\begin{theorem} \label{MazOvresult}
	Let $(A, \textbf{e})$ be a properly stratified algebra having a duality with characteristic tilting module $T$ coinciding with the characteristic cotilting module.
	Then $\findim(A)= 2 \pd(T)$.
\end{theorem}
\begin{proof}
	See \cite{MazOv}.
\end{proof}

\section{A characterisation of Gorenstein properly stratified algebras}
We will need the following lemma on local algebras in this section:
\begin{lemma} \label{localalgebralemma}
Let $A$ be a local algebra.
\begin{enumerate}
\item $A$ has finitistic dimension equal to zero.
\item $A$ is Gorenstein if and only if $A$ is selfinjective.

\end{enumerate}

\end{lemma}
\begin{proof}
Note that in a local algebra there is a unique indecomposable projective module and thus all projective modules have the same Loewy length.
Let $M$ be a non-projective module and let $f_0$ be the projective cover $f_0: P \rightarrow M$. Then $\Omega^1(M)=\ker(f_0)$ has Loewy length strictly smaller than the Loewy length of $P$ and thus $\Omega^1(M)$ cannot be projective. Therefore all $\Omega^i(M)$ are non-projective and thus $M$ has infinite projective dimension and the finitistic dimension is zero. This shows (1).
Clearly any selfinjective algebra is Gorenstein.
Now assume that $A$ is a local Gorenstein algebra. Since $A$ is Gorenstein, the finitistic dimension of $A$ coincides with the Gorenstein dimension of $A$ by \ref{auslander-reitenfindimresult}. Thus the Gorenstein dimension is zero and $A$ is selfinjective, which shows (2).
\end{proof}

\begin{theorem} \label{mainresult}
Let $(A, \textbf{e})$ be a properly stratified algebra with characteristic tilting module $T$ and characteristic cotilting module $C$. Then the following are equivalent:
\begin{enumerate}
\item $A$ is Gorenstein.
\item $T=C$.
\item $\widehat{\add(T)} = \F (\nabla)$.
\end{enumerate}

\end{theorem}

\begin{proof}
We first show that (2) is equivalent to (3) and that (3) implies (1).
Assume (2). Then $T=C$ and thus $\widehat{\add(T)}=\widehat{\add(C)}=\F( \nabla)$.
Now assume (3), so that we have $\widehat{\add(T)} = \F (\nabla)$.
Since $A$ is properly stratified we have $D(A) \in \F (\nabla)=\widehat{\add(T)}.$
Thus there is by definition an exact sequence 
$$0 \rightarrow T_r \rightarrow \cdots \rightarrow T_0 \rightarrow D(A) \rightarrow 0$$
with $T_i \in \add(T)$. 
Now since $T$ is a tilting module, it has finite projective dimension.
Let $K_{i+1}$ be defined as the kernel of the maps $T_i \rightarrow T_{i-1}$ for $i \geq 0$ with $T_{-1}:=D(A)$.
Then the exact sequence $0 \rightarrow T_r \rightarrow T_{r-1} \rightarrow K_{r-1} \rightarrow 0$ shows that $K_{r-1}$ has finite projective dimension since $T_r, T_{r-1} \in \add(T)$ have finite projective dimension. 
Inductively we see that also each $K_{i}$ has finite projective dimension because of the short exact sequences $0 \rightarrow K_{i+1} \rightarrow T_i \rightarrow K_i \rightarrow 0$ for each $i \geq 1$.
Thus also $D(A)$ has finite projective dimension because of the short exact sequence $0 \rightarrow K_1 \rightarrow T_0 \rightarrow D(A) \rightarrow 0$.
Since the right module $D(A)$ has finite projective dimension, by duality the left module $A$ has finite injective dimension. Since $A$ is assumed to be properly stratified, $A$ and $A^{op}$ have finite finitistic dimension by \ref{finitisticbound} and thus by \ref{auslander-reitenfindimresult} $A$ is Gorenstein.
Thus we see already that (3) implies (1) and we go further that (3) implies also (2).
To see this, note that since $A$ is Gorenstein we know that the tilting module $T$ is also a cotilting module by \ref{gorenstein characterisation}.
By the cotilting correspondence \ref{cotiltcorres}, $T$ as a cotilting module corresponds to the cotorsion pair $( {}^{\perp}T, \widehat{\add(T)})$. But the characteristic cotilting module $C$ also has $\widehat{\add(T)}=\F(\nabla)$ as the right side in the cotorsion pair (and the left and right sides in a cotorsion pair determine each other) and this implies $T=C$.

Assume now that (1) holds. The characteristic cotilting module is the unique module (up to multiplicities)  satisfying $\add C=\F(\pSt)\cap \F(\Cs)$. Thus, it is enough to prove that $\add T=\F(\pSt)\cap \F(\Cs)$. The category $\F(\pSt)\cap \F(\Cs)$ only contains $n$ non-isomorphic indecomposables modules, so we just need to show that $T\in \F(\pSt)\cap \F(\Cs)$. Since $T$ is the characteristic tilting module $\add T=\F(\St)\cap \F(\pCs)$. By Lemma \ref{filtrations of standard in proper standard}, $T\in \F(\St)\subset \F(\pSt)$. Due to $T\in \F(\pCs)$ and Lemma 1.2 (vi) of \cite{ADLU}, $\Ext^{j>0}_A(\St(i), T)=0$, for all $i$. In view of Theorem 1.6 of \cite{ADLU} and $A^{op}$ being standardly stratified, we want
\begin{align*}
	0=\Ext_{A^{op}}^1(DT, \pCs_{op}(i))\cong\Ext_A^1(D\pCs_{op}(i), T)\cong \Ext_A^1(\pSt(i), T), 
\end{align*}for all $i$.
 By Theorem \ref{gorenstein characterisation}, $T$ has finite injective dimension. Let $s_i$ be the smallest non-negative integer such that $\Ext_A^{s_i}(\pSt(i), T)\neq 0$ and $\Ext_A^{s_i+1}(\pSt(i), T)=0$, for each $i$.  Let $i=\{1, \ldots, n\}$ be arbitrary and assume that $s_i>0$.
 Since $A$ is properly stratified then $\St(i)$ admits a filtration $0=X_0\subset X_1\subset \cdots \subset X_t=\St(i)$ with factors of the form $\pSt(i)$ by Lemma \ref{filtrations of standard in proper standard}.  Applying $\Hom_A(-, T)$ to this filtration yields the exact sequence
 \begin{align}
 	\Ext_A^{s_i}(X_r, T)\rightarrow \Ext_A^{s_i}(X_{r-1}, T)\rightarrow \Ext_A^{s_i+1}(\pSt(i), T)=0, \ r=1, \ldots, t. \label{eq1}
 \end{align}
Since $X_1=\pSt(i)$ it follows by (\ref{eq1}) that $\Ext_A^{s_i}(X_r, T)\neq 0$, for all $r\geq 1$. In particular, $0\neq \Ext_A^{s_i}(X_t, T)=\Ext_A^{s_i}(\St(i), T)$. This contradicts $T$ belonging to $\F(\pCs)$. Thus, $s_i=0$. As $i$ was arbitrary, this shows that $T\in \F(\Cs)$.
\end{proof}

\begin{remark}
	Given a characteristic tilting module $T$ of a properly stratified algebra $A$, to check condition (2) of Theorem \ref{mainresult} it is enough to observe whether $T\in \mathcal{F}(\Cs)$. For a characteristic cotilting module $C$ of a properly stratified algebra $A$, it is enough to observe that $C\in \mathcal{F}(\St)$.
\end{remark}

\subsection{Properties of Gorenstein properly stratified algebras}

We record some further properties of Gorenstein properly stratified algebras. Recall that a module $M$ over a Gorenstein algebra $A$ is called \emph{Gorenstein projective} (or also maximal Cohen-Macaulay in the literature) if $\Ext_A^i(M,A)=0$ for all $i>0$. Dually $N$ is called \emph{Gorenstein injective} if $\Ext_A^i(D(A),N)=0$ for all $i >0$.
\begin{proposition}\label{projectiveisproperlystratified}
Let $(A, \textbf{e})$ be a properly stratified Gorenstein algebra.
\begin{enumerate}
\item Every Gorenstein injective module is in $\F(\overline{\nabla})$.
\item Every Gorenstein projective module is in $\F(\overline{\Delta})$.
\item $P^{< \infty} \cap \F(\overline{\Delta})=\F(\Delta)$.
\item $I^{< \infty } \cap \F(\overline{\nabla})=\F(\nabla)$.

\end{enumerate}

\end{proposition}
\begin{proof}
We prove (1), the proof of (2) is dual.
Let $T$ be a tilting module. $T$ has by definition finite projective dimension and since $A$ is Gorenstein, $T$ also has finite injective dimension. 
Let 
$$0 \rightarrow T \rightarrow I_0 \rightarrow \cdots \rightarrow I_{r+1} \rightarrow 0$$
be a minimal injective coresolution of $T$.
Then we have the short exact sequence $0 \rightarrow \Omega^{-r}(T) \rightarrow I_r \rightarrow I_{r+1} \rightarrow 0.$
Now let $X$ be an $A$-module with $X \in D(A)^{\perp}$, then we apply the functor $\Hom_A(-,X)$ to the above short exact sequence and obtain the following exact sequence as a part of the resulting long exact Ext-sequence for every $t \geq 1$:
$$\cdots \rightarrow \Ext_A^t(I_{r+1},X) \rightarrow \Ext_A^t(I_r,X) \rightarrow \Ext_A^t(\Omega^{-r}(T),X) \rightarrow \Ext_A^{t+1}(I_{r+1},X) \rightarrow \cdots .$$
Since we assume $X \in D(A)^{\perp}$, we obtain $\Ext_A^t(\Omega^{-r}(T),X)=0$ for all $t \geq 1$.
Now we apply the functor $\Hom_A(-,X)$ to the short exact sequence $0 \rightarrow \Omega^{-(r-1)}(T) \rightarrow I_{r-1} \rightarrow \Omega^{-r}(T) \rightarrow 0$ conclude that also $\Ext_A^t(\Omega^{-(r-1)}(T),X)=0$ for all $t \geq 1$ and by using induction we obtain in the same way $\Ext_A^t(T,X)=0$ for all $t \geq 1$.
Thus we have $D(A)^{\perp} \subseteq T^{\perp}$ for every tilting module $T$ assuming that $A$ is Gorenstein. Choosing for $T$ the characteristic tilting module, we obtain $D(A)^{\perp} \subseteq F(\overline{\nabla})$. \newline
We prove now (3), the proof of (4) is dual.
$\F(\Delta) \subseteq P^{< \infty} \cap \F(\overline{\Delta})$ is clear by \ref{propstratresults1}.
Now assume $X \in P^{< \infty} \cap \F(\overline{\Delta})$.
By the dual of Lemma 2.5. (iii) of \cite{ADLU}, every module $M \in \F(\overline{\Delta})$ has a (possibly infinite) $\add(T)$-coresolution. 
Thus $X$ has such a coresolution as follows:
$$0 \rightarrow X \rightarrow T_0 \rightarrow T_1 \rightarrow \cdots \rightarrow T_i \rightarrow \cdots ,$$
with $T_i \in \add(T)$.
Now since $A$ is Gorenstein, $X \in P^{< \infty}$ has finite injective dimension.
If the $add(T)$-coresolution would be infinite, we would have $\Ext_A^i(K_i,X) \neq 0$ for all $i >0$ when $K_i$ is the kernel of $T_i \rightarrow T_{i+1}$ and $X$ would have infinite injective dimension. Thus the coresolution is finite and $X \in \widecheck{\add(T)}=\F(\Delta).$
\end{proof}

\begin{theorem}\label{endofstandardisfrobenius}
Let $(A, \textbf{e})$ be a properly stratified Gorenstein algebra.
Then the following hold for any $i$:
\begin{enumerate}
\item $\End_A(\Delta(i))$ is a Frobenius algebra.
\item $\End_A(\nabla(i))$ is a Frobenius algebra.
\item The quotient $(A/A(e_{i+1}+\cdots+ e_n)A, \textbf{e}^i)$ is a properly stratified Gorenstein algebra, with  $\textbf{e}^i=\{e_1^i, \ldots, e_i^i \}$ where $e_j^i$ denotes the image of $e_j$ in $A/A(e_{i+1}+\cdots+ e_n)A$.
\end{enumerate}
\end{theorem}
\begin{proof}
We prove (1), the proof of (2) is dual.
Assume $A$ is properly stratified Gorenstein with linear order $1,...,n$ of the $n$ simple $A$-modules.
Let $I$ be an injective $A$-module. Since $\F(\Delta)$ is contravariantly finite, there exists a minimal right $\F(\Delta)$-approximation of $I$:
$$0 \rightarrow K_I \rightarrow F_0 \rightarrow I \rightarrow 0,$$ 
where $F_0 \in \F(\Delta)$.
Since $\F(\Delta)$ is extension-closed we obtain $\Ext_A^1(\F(\Delta),K_I)=0$ from Wakamatsu's Lemma.
Thus $K_I \in \F(\overline{\nabla})$.
Now we apply the functor $\Hom_A(X,-)$ for $X \in \F(\Delta)$ to the above short exact sequence to obtain from the long exact $\Ext$-sequence:
$$\cdots \rightarrow 0=\Ext_A^1(X,K_I) \rightarrow \Ext_A^1(X,F_0) \rightarrow \Ext_A^1(X,I)=0 \rightarrow \cdots .$$
Thus $\Ext_A^1(X,F_0)=0$ for all $X \in \F(\Delta)$ and thus $F_0 \in \F(\Delta) \cap \F( \overline{\nabla})=\add(T)$.
Since $K_I \in \F(\overline{\nabla})$ we can continue this process to obtain an $\add(T)$-resolution of $I$ as follows:
$$\cdots \rightarrow F_t \rightarrow F_{t-1} \rightarrow \cdots \rightarrow F_0 \rightarrow I \rightarrow 0.$$
Let $K_t$ denote the kernel of the map $F_t \rightarrow F_{t-1}$ then if this resolution would be infinite, we would have $\Ext^t(I,K_t) \neq 0$ for all $t \geq 1$ and thus $I$ would have infinite projective dimension, contradicting our assumption that $A$ is Gorenstein.
Thus there exists a finite $\add(T)$-resolution of $I$:
$$0 \rightarrow T_r \rightarrow T_{r-1} \rightarrow \cdots \rightarrow T_0 \rightarrow I \rightarrow 0,$$
with $T_i \in \add(T)$ for all injective $A$-modules $I$.
Since $\Delta(n)$ is projective, the module $I:=D\Hom_A(\Delta(n),A)$ is injective.
We apply the functor $\Hom_A(\Delta(n),-)$ to the above finite $\add(T)$-resolution of this $I$ to obtain the exact sequence:
\begin{equation}\label{exactsequ1}
0 \rightarrow \Hom_A(\Delta(n),T_r) \rightarrow \cdots \rightarrow \Hom_A(\Delta(n),T_0) \rightarrow \Hom_A(\Delta(n),I) \rightarrow 0.
\end{equation} 
By \cite[Lemma 2.5]{ADLU}, there is a left $\F(\overline{\nabla})$-approximation of $\Delta(i)$ as follows:
$$0 \rightarrow \Delta(i) \rightarrow T(i) \rightarrow Y_i \rightarrow 0,$$
where $Y_i \in \F(\Delta_{j<i})$ and $T(i)$ is the indecomposable $i$-th summand of the characteristic tilting module $T$.
We apply $\Hom_A(\Delta(n),-)$ to the left approximation as above for $i=n$ to obtain:
$$0 \rightarrow \Hom_A(\Delta(n),\Delta(n)) \rightarrow \Hom_A(\Delta(n),T(n)) \rightarrow \Hom_A(\Delta(n),Y_n)=0 \rightarrow 0.$$
Thus $\Hom_A(\Delta(n),T(n)) \cong \Hom_A(\Delta(n),\Delta(n)) $ is a projective $R$-module, where $R:=\End_A(\Delta(n))$.
Since $\Hom_A(\Delta(n),T(j))=0$ for $j < n$, we have that $\Hom_A(\Delta(n),T')$ is a projective $R$-module for any $T' \in \add(T)$.
Thus the exact sequence \ref{exactsequ1} gives us that the module $\Hom_A(\Delta(n),I)$ has finite projective dimension as an $R$-module. But since $R$ is local and therefore has finitistic dimension equal to zero by \ref{localalgebralemma} (1), $\Hom_A(\Delta(n),I)$ must be even projective.
Now, since $\St(n)$ is $A$-projective
$$\Hom_A(\Delta(n),I)=\Hom_A(\Delta(n),D\Hom_A(\Delta(n),A)) \cong D(\Hom_A(\St(n), A)\otimes_A \St(n))=DR$$
is injective as an $R$-module.
Thus $\Hom_A(\Delta(n),I)$ is a projective-injective non-zero module for the local algebra $R$ and thus $R$ is selfinjective and therefore also a Frobenius algebra (using that every selfinjective local algebra is automatically Frobenius).
We proved that $R=\End_A(\Delta(n))$ is Frobenius and now we can use $\Ext_{A/Ae_n A}^i(M,N) \cong \Ext_A^i(M,N)$ for every $A/Ae_n A$-modules $M$ and $N$, which is true by \ref{extlemmaidemideal}.
This gives us that $\Ext_{A/Ae_n A}^i(M,A/Ae_n A) \cong \Ext_A^i(M,A/Ae_n A)=0$ for all $i>p$ for some finite number $p$ since as an $A$-module $A/Ae_n A$ has finite projective dimension (since $Ae_nA$ has finite projective dimension) and therefore also finite injective dimension, using that $A$ is Gorenstein.
Thus with $A$, also the properly stratified algebra $A/Ae_n A$ is Gorenstein. Now we can use induction to conclude that $\End_A(\Delta(i))$ is Frobenius for all $i=1,...,n$ and the remaining cases of (3).
\end{proof}

The next example shows that there are non-Gorenstein properly stratified algebras such that $P^{< \infty} \cap \F(\overline{\Delta}) = \F( \Delta)$ and such that all endomorphism rings of the $\Delta(i)$ are Frobenius algebras. This shows that those properties can not be used to characterise the Gorenstein property for properly stratified algebras.

\begin{example}
Let $Q$ be the quiver
\begin{tikzcd}
	2 \arrow[r, "\beta"] & 1 \arrow[loop, distance=3em, out=35, in=-35, "\alpha"]
\end{tikzcd}
and $A=KQ/I$ with $I=<\alpha^2,\beta \alpha>$.
$A$ is a radical square zero algebra and it is easily checked that $A$ has exactly five indecomposable modules with $\add(P^{< \infty})=\add(e_1 A \oplus e_2 A)$.  Thus $A$ has finitistic dimension equal to zero. Since for Gorenstein algebras, the finitistic dimension coincides with the Gorenstein dimension, $A$ can only be Gorenstein when it is selfinjective. Since the indecomposable projective modules are not injective, $A$ is not selfinjective and therefore also not Gorenstein.
We have $\Delta(1)=e_1 A, \Delta(2)=e_2 A, \overline{\Delta(1)}=S_1$ and $\overline{\Delta(2)}=e_2 A$.
Thus $A$ is properly stratified and not Gorenstein.
We have $\End_A(\Delta(1))\cong K[x]/(x^2)$ and $\End_A(\Delta(2)) \cong K$, which are Frobenius algebras.
We have $P^{< \infty} \cap \F(\overline{\Delta}) = \F( \Delta)$ since $\add(P^{< \infty})=\add(e_1 A \oplus e_2 A)=\F(\Delta).$

\end{example} 

The following example shows that unlike quasi-hereditary algebras, in general, Gorenstein properly stratified algebras have no recollements of bounded derived categories for $A/AeA$, $A$ and $eAe$.

\begin{example}
Let $K$ be an algebraically closed field and let $A$ be the bound quiver $K$-algebra
$\begin{tikzcd}[every arrow/.append style={shift left=0.75ex}]
	1 \arrow[r, "\alpha_1"] & 2 \arrow[l, "\beta_1"] \arrow[r, "\alpha_2"] & 3 \arrow[l, "\beta_2"]
\end{tikzcd}$ with relations $0=\beta_1\alpha_1\alpha_2=\beta_2\beta_1\alpha_1=\beta_2\alpha_2$, $\alpha_2\beta_2=\beta_1\alpha_1\beta_1\alpha_1$.
\end{example}$A$ is a properly stratified algebra with $\St(3)=\pSt(3)=e_3A$, $\St(1)=\pSt(1)=S(1),$ $\St(2)=e_2A/\St(3)$, and $\pSt(2)=\begin{array}{c}
2\\1
\end{array}$. $A$ has injective dimension two and $A/Ae_3A$ is the bound quiver algebra
$\begin{tikzcd}[every arrow/.append style={shift left=0.75ex}]
	1 \arrow[r, "\alpha_1"] & 2 \arrow[l, "\beta_1"] 
\end{tikzcd}$ with relation $\beta_1\alpha_1\beta_1\alpha_1=0$. Note that $\Hom_A(A/Ae_3A, P(3))\cong  \pSt(2)$ as $A/Ae_3A$-modules. Since $\pSt(2)$ has infinite injective dimension it follows that the condition (1) of Proposition 3.4 \cite{CheKo} holds but condition (3) fails.
\section{GIGS algebras}
We first define a generalised version of the algebras studied by Fang and Koenig in \cite{FanKoe3}.
\begin{definition}
A \emph{GIGS} algebra is a gendo-symmetric properly stratified Gorenstein algebra $(A, \textbf{e})$ having a duality.
 
\end{definition}
In \cite{FanKoe3}, the class of gendo-symmetric and quasi-hereditary algebras $(A, \textbf{e})$ having a duality is called $\mathcal{A}$.
Note that a GIGS algebra has finite global dimension if and only if it is in the class $\mathcal{A}$ of \cite{FanKoe3}. By \ref{MazOvresult}, GIGS algebras have even Gorenstein dimension equal to two times the projective dimension of the characteristic tilting module since the Gorenstein dimension coincides with the finitistic dimension.
Our main result \ref{mainresult} says that being Gorenstein is equivalent to the characteristic tilting module coinciding with the characteristic cotilting module, which is the central property used by Fang and Koenig in \cite{FanKoe3} to prove their main results. For a properly stratified algebra $A$, we will denote by $R_A$ the Ringel dual of $A$. Note that since being Gorenstein is a derived invariant, $A$ is Gorenstein if and only if $R_A$ is Gorenstein.
We give an example of a GIGS algebra with infinite global dimension.
\begin{example}
Let $A=K[x,y]/(xy,x^2-y^3)$ and $M=A \oplus xA \oplus yA \oplus x^2 A$ and $B=\End_A(M)$.
Note that $A$ is a (commutative) symmetric Frobenius algebra and thus $B$ is gendo-symmetric. By results of Chen and Dlab in \cite{CheDl} $B$ is also properly stratified with a duality.
$B$ has Gorenstein dimension 4, dominant dimension 2 and infinite global dimension. 

\end{example}

The next result was proven in \cite[Theorem 4.3]{FanKoe3} for GIGS algebras with finite global dimension, but the arguments remain valid  for GIGS algebras once we establish that the Ringel dual is also gendo-symmetric.
\begin{theorem} \label{fankoetheorem}
Let $(A, \textbf{e})$ be a GIGS algebra with characteristic tilting module $T$.
\begin{enumerate}
\item The Ringel dual of $A$, $R_A$, is again a GIGS algebra.
\item $2 \domdim(T)=\domdim(A)$.
\end{enumerate}
\end{theorem}
\begin{proof}
	By Theorem \ref{mainresult}, $T$ is also the characteristic cotilting module and therefore the Ringel dual of $A$, $R_A$, is a properly stratified algebra by Theorem 5 of \cite{FrMa}. Using the exact sequences for tilting modules established in \cite[Lemma 2.5]{ADLU} we can see that $T(i)^\natural \cong T(i)$, for all $i$. In particular, $T\cong T^\natural$. By Proposition 2.4 of \cite{FanKoe3}, $R_A$ has a duality $\tau$
	which fixes all the idempotents $T\twoheadrightarrow T(i)\hookrightarrow T$. For the statement (1), it remains to show that $R_A$ is gendo-symmetric. Let $eA$ be the minimal faithful projective-injective module of $A$. In particular $(eA)^\natural\cong eA$. Analogously to \cite[Lemma 4.2]{FanKoe3}, $\Hom_A(T, eA)$ is a right minimal faithful projective-injective module over $R_A$ and $\Hom_A(DT, Ae)$ is a left minimal faithful projective-injective module over $R_A$. Further, as $(eAe, R_A)$-bimodules
	\begin{align}
		\Hom_{R_A}(\Hom_A(T, eA), R_A)&\cong \Hom_{R_A}(\Hom_A(T, eA), \Hom_A(T, T))\\&\cong \Hom_A(eA, T)\cong \Hom_A(DT, D(eA))\cong \Hom_A(DT, Ae).
	\end{align} The last isomorphism follows from $A$ being gendo-symmetric. Note also that since $\Hom_A(T, -)$ is fully faithful on $\F(\Cs)$, we also have the isomorphism $eAe\cong \End_A(eA)\cong \End_{R_A}(\Hom_A(T, eA))$. Since Lemma 4.2 of  \cite{FanKoe3} remains true for properly stratified algebras, part (2) of the proof of Theorem 4.3 of \cite{FanKoe3} also holds for $R_A$. That is, $\domdim R_A\geq 2$. By Theorem 3.2 of \cite{FanKoe2}, $R_A$ is gendo-symmetric and (1) holds.  It is now clear that the analogous statement of \cite[Lemma 3.2(5)]{FanKoe3} holds for $R_A$. Finally, since the remaining arguments that are involved in the proof of  \cite[Theorem 4.3]{FanKoe3} remain true under the assumption of $(A, \textbf{e})$ being properly stratified in place of just being quasi-hereditary, (2) follows.
\end{proof}

We can now apply our results to calculate the Ringel dual of GIGS algebras that are minimal Auslander-Gorenstein algebras. Recall that an algebra $A$ is called \emph{minimal Auslander-Gorenstein} if it has finite Gorenstein dimension equal to the dominant dimension and both dimensions are at least two (we exclude selfinjective algebras in the definition for minimal Auslander-Gorenstein algebras here for simplicity as selfinjective algbras are not interesting for standardly stratified algebras). Minimal Auslander-Gorenstein algebras generalise higher Auslander algebras, which are exactly those minimal Auslander-Gorenstein algebras with finite global dimension. 

We will need the following result on minimial Auslander-Gorenstein algebras, see \cite[Theorem 4.5]{IyaSol}.

\begin{theorem}
There is a bijection between Morita equivalence classes of minimal Auslander-Gorenstein algebras $A$ of Gorenstein dimension $n \geq 2$ and equivalence classes of tuples $(B,N)$ of algebras $B$ with a generator-cogenerator $N$ of $\mod B$ such that $\Ext_B^i(N,N)=0$ for $i=1,...,n-2$ and $\tau(\Omega^{n-2}(N)) \in \add(N)$.
The bijection associates to $(B,N)$ the algebra $A=\End_B(N)$.

\end{theorem}
Note that when $B$ is symmetric and $N=B \oplus M$ with $M$ having no projective direct summands as in the previous theorem, then $\tau \cong \Omega^2$ (see for example \cite[Proposition 3.8 in Chapter IV]{ARS}) and the condition $\tau(\Omega^{n-2}(N)) \in \add(N)$ simplifies to $\Omega^{n}(M) \cong M$.

\begin{theorem} \label{Ringeldualtheorem}
Let $A=\End_U(U \oplus M)$ be a GIGS algebra with a symmetric algebra $U$ and a generator $U \oplus M$ of $\mod U$, where we can assume that $M$ has no projective direct summands.
Assume $A$ is furthermore a minimal Auslander-Gorenstein algebra with Gorenstein dimension equal to $2d$ for some $d \geq 1$.
Then the Ringel dual $R_A$ of $A$ is isomorphic to $\End_U(U \oplus \Omega^d(M))$. In particular, $R_A$ is again a GIGS algebra that is minimal Auslander-Gorenstein with Gorenstein dimension $2d$.

\end{theorem}
\begin{proof}
$A$ is properly stratified and Gorenstein and thus we can apply our main result \ref{mainresult} to conclude that the characteristic tilting module of $A$ coincides with the characteristic cotilting module, which we denote by $T$.
By \cite[Theorem 2.5 and Theorem 2.6.]{Mar}, $T \cong eA \oplus \Omega^{-d}(A)$ when $eA$ denotes the minimal faithful projective-injective $A$-module.
Now $R_A \cong \End_A(eA \oplus \Omega^{-d}(A))$.
Since $d \geq 1$, the module $T$ has dominant dimension at least one and codominant dimension at least one.
By \cite[Lemma 3.1 (2) (ii)]{APT} the functor $\Hom_A(eA,-)$ induces an isomorphism of algebras
between $\End_A(eA \oplus \Omega^{-d}(A))$ and $\End_{eAe}(eAe \oplus \Omega^{-d}(A)e)$.
Now since the functor $\Hom_A(eA,-)$ is exact and sends projective-injective $A$-modules to injective $eAe$-modules, we have $\Omega^{-d}(A)e \cong \Omega^{-d}(Ae)$. Thanks to the Morita-Tachikawa correspondence and since $eAe$ is symmetric, we have $Ae \cong U \oplus M$ as $(A, eAe)$-bimodules and thus 
$$R_A \cong \End_A(eA \oplus \Omega^{-d}(A)) \cong \End_{eAe}(eAe \oplus \Omega^{-d}(Ae))\cong \End_U(U \oplus \Omega^{-d}(M)).$$
Since $eAe$ is symmetric and $A$ is minimal Auslander-Gorenstein, $M$ is $2d$-periodic: $\Omega^{2d}(M) \cong \tau(\Omega^{2d-2}(M)) \cong M$ and thus $\End_U(U \oplus \Omega^{-d}(M)) \cong \End_U(U \oplus \Omega^{d}(M))$.
When $B$ is a general symmetric algebra and $N$ a $B$-module without projective direct summands, the algebras $\End_B(B \oplus N)$ and $\End_B(B \oplus \Omega^i(N))$ are almost $\nu$-stable derived equivalent in the sense of \cite{HX} and this derived equivalence preserves the Gorenstein and dominant dimensions, see \cite[Corollary 1.3 (2)]{HX} and \cite[Corollary 1.2]{HX}.
Now, the Ringel dual has also a duality and is properly stratified  by Theorem \ref{fankoetheorem} and thus the statement follows.
\end{proof}

We have the following corollary:
\begin{corollary} \label{ringelselfdualcorollary}
Let $A$ be a GIGS  basic algebra that is minimal Auslander-Gorenstein.
Then $A$ is Ringel self-dual if and only if $A$ is isomorphic to its Ringel dual $R_A$.

\end{corollary}

\begin{proof}
For any properly stratified algebra $A$, being Ringel self-dual implies that $A$ and $R_A$ are, in particular, Morita equivalent. Since both algebras are basic, $A$ and $R_A$ are isomorphic.
Now assume that $A$ and $R_A$ are isomorphic, then the isomorphism induces a Morita equivalence between $\mod A$ and $\mod R_A$. Now by \cite[Main theorem]{Mar} a GIGS algebra $C$ that is minimal Auslander-Gorenstein has the property that $\mathcal{F}(\Delta)$ is equal to the subcategory of $\mod C$ consisting of the modules with projective dimension at most $\frac{r}{2}$ when $r$ is the Gorenstein dimension of $C$.
Since by \ref{Ringeldualtheorem} $A$ and $R_A$ have the same Gorenstein dimension and a Morita equivalence preserves the projective dimension of modules, the subcategories $\mathcal{F}(\Delta^A)$ and $\mathcal{F}(\Delta^{R_A})$ are isomorphic and $A$ is Ringel self-dual. 
\end{proof}

\subsection{Centraliser algebras of nilpotent matrices}\label{Centraliser algebras of nilpotent matrices}

Let $H$ be an arbitrary nilpotent matrix in $K^{n \times n}$ for a field $K$ and $n \geq 2$.
Then it is well known, see for example \cite[Proposition 4.1]{AW}, that the centraliser $\{ X \in K^{n \times n} \mid XH=HX \}$ is isomorphic to $\Hom_{K[x]}(V_H, V_H)$, where we view $V_H \cong K^n$ as an $K[x]$-module by letting $x$ act as $H$ on $K^n$.
Now since $H$ is nilpotent $x^n$ annihilates $V_H$ and thus $\Hom_{K[x]}(V_H, V_H) \cong \Hom_{K[x]/(x^n)}(V_H, V_H)$ is a finite dimensional $K$-algebra. Since every indecomposable $K[x]/(x^n)$-module is isomorphic to a module of the form $(x^k)/(x^n)$ for some $k=0,1,...,n-1$ we can assume that $V_H$ is a direct sum of such indecomposable modules and we can up to Morita equivalence assume that $V_H$ is basic. Furthermore, we can assume that $H$ is nilpotent with $H^n=0$ and $H^{n-1} \neq 0$ and thus that the indecomposable module $K[x]/(x^n)$ (which is the unique indecomposable projective $K[x]/(x^n)$-module) appears as a direct summand.
We call $\Hom_{K[x]}(V_H, V_H) \cong \Hom_{K[x]/(x^n)}(V_H, V_H)$ the \emph{centraliser algebra} of the nilpotent matrix $H$. 
A classical theorem of Frobenius gives the vector space dimension of a general centraliser algebra of a matrix. For this and more on centraliser algebras we refer for example to \cite[Section 5.5]{AW}.
We will see that this algebra satisfies the assumption of \ref{Ringeldualtheorem} and we will use this to determine when the algebra is Ringel self-dual.

We now fix notation for our problem.
Let $U=K[x]/(x^n)$ and $N=U \oplus M$ an arbitrary non-projective generator of $\mod U$, where we can assume that $M$ has no projective direct summands and up to Morita equivalence we can also assume that $N$ is basic. The algebra $A=\End_U(N)$ is properly stratified (see \cite[Theorem 2.4]{CheDl}) and has a duality (this is explained in \cite{CheDl} at the end of page 64). 

\begin{example}
Let $U=K[x]/(x^3)$ and $N=U \oplus S$, where $S$ is the unique simple $U$-module.
Then $A=\End_U(N)$ is isomorphic to the following quiver algebra $KQ/I$ with quiver $Q$:
\begin{tikzcd}
	1 \arrow[r, "\beta_1", shift left=0.75ex] & 2 \arrow[loop, distance=3em, out=35, in=-35, "\beta_3"] \arrow[l, "\beta_2", shift left=0.75ex]
\end{tikzcd}
and relations $I=\langle \beta_1 \beta_2 , \beta_1 \beta_3 , \beta_3 \beta_2 , \beta_2 \beta_1 - \beta_3^2 \rangle$. $A$ has infinite global dimension and thus is not quasi-hereditary.
\end{example}

\begin{proposition}
$A=\End_U(N)$ is a GIGS algebra that is minimal Auslander-Gorenstein of Gorenstein dimension 2. It is Ringel self-dual if and only if $M \cong \Omega^1(M)$.

\end{proposition}
\begin{proof}
That $A$ is properly stratified with a duality follows from \cite{CheDl}. Now $A$ is minimal Auslander-Gorenstein of Gorenstein dimension 2 since $M$ is 2-periodic. 
Thus $A$ is a GIGS algebra that is Auslander-Gorenstein and we can use \ref{ringelselfdualcorollary} that shows that $A$ is Ringel self-dual if and only if $A$ is isomorphic to its Ringel dual.
By \ref{Ringeldualtheorem}, the Ringel dual of $A$ is isomorphic to $\End_U(U \oplus \Omega^1(M))$.
Thus when $M \cong \Omega^1(M)$, $A$ is Ringel self-dual. \newline
Now assume that $A=\End_U(N)$ is Ringel self-dual. We show that then $M \cong \Omega^1(M)$.
Let $M$ have indecomposable non-projective direct summands $U/J^{p_i}$ for $i=0,...,r$ with $p_0<p_1<...<p_r$, where we recall that $J$ denotes the Jacobson of the algebra $U$.
Then $\Omega^1(M)$ has indecomposable non-projective direct summands $U/J^{n-p_i}$.
If $B_1:=\End_U(U \oplus M)$ and $B_2:=\End_U(U \oplus \Omega^1(M))$ are isomorphic, which is equivalent to being Ringel self-dual, they have the same dimensions of their indecomposable projective modules.
Note that we have in general $\Hom_U(U/J^k, U/J^t) \cong J^{\max(0,t-k)}/J^t$ and thus $\dim(\Hom_U(U/J^k, U/J^t))=\dim J^{\max(0,t-k)}/J^t=\min(k,t)$. 
The indecomposable projective $B_1$-modules are given by $P_{B_1}^i=\Hom_U(U \oplus M , U/J^{p_i})$.
Similarly, the indecomposable projective $B_2$-modules are given by $P_{B_2}^i=\Hom_U(U \oplus \Omega^1(M), U/J^{n-p_{r-i}})$.
Here we choose the ordering for the indecomposable projective $B_1$ and $B_2$ modules in increasing order of their vector space dimensions. When $B_1$ and $B_2$ are isomorphic, we have $\dim(P_{B_1}^i)=\dim(P_{B_2}^i)$ for $i=0,...,r$.
Now $\dim\Hom_U(U \oplus M , U/J^{p_i})=\dim \Hom_U(U, U/J^{p_i})+ \dim \Hom_U(M, U/J^{p_i})=p_i+ \sum\limits_{k=0}^{r}{\min(p_k,p_i)}.$
In the same way we obtain $\dim \Hom_U(U \oplus \Omega^1(M), U/J^{n-p_{r-i}})=n-p_{r-i}+ \sum\limits_{s=0}^{r}{\min(n-p_s,n-p_{r-i})}.$
The condition $\dim(P_{B_1}^0)=\dim(P_{B_2}^0)$ gives us that $p_0+(r+1)p_0=n-p_r+(r+1)(n-p_r)$ and thus $p_0=n-p_r$ and using induction and the equations $\dim(P_{B_1}^i)=\dim(P_{B_2}^i)$ we conclude that $p_i=n-p_{r-i}$ for all $i=0,...,r$.
This is equivalent to $M= \Omega^1(M)$.
\end{proof}

Thus in the situation of the previous algebra, we have Ringel self-duality if and only if the generators are  stable 1-periodic.

\subsection{Representation-finite blocks of Schur algebras}
The next application shows that the representation-finite blocks of Schur algebras are Ringel self-dual even when the generator is not stable $d$-periodic.
For statements without proofs about the representation-finite blocks of Schur algebras, we refer for example to \cite[section 4]{Mar}.
Recall that the representation-finite blocks of Schur algebras are given by quiver and relations as 
\[
\begin{tikzcd}[every arrow/.append style={bend left}]
	1 \arrow[r, "a_1"]  & 2 \arrow[l, "b_1"] \arrow[r, "a_2"] & 3 \arrow[l, "b_2"] \arrow[r, "a_3"] & 
	4 \arrow[l, "b_3"] & \cdots & m-1 \arrow{r}[above]{a_{m-1}} & m \arrow{l}[below]{b_{m-1}}
\end{tikzcd}, \quad m\geq 1,
\] $b_{m-1} a_{m-1}, \ b_{i-1}a_{i-1}-a_i b_i,\ a_{i-1}a_i, \ b_i b_{i-1}, \ i=2, \ldots, m-1$.
We denote the quiver algebra of a representation-finite block of a Schur algebra with $n$ simples by $\mathcal{A}_n$. By $P(i)$ we mean the projective indecomposable of $\mathcal{A}_n$ associated to the vertex $i$. We denote by $\mathcal{B}_n$ the endomorphism algebra $\End_{A_{n+1}}(P(1)\oplus \cdots \oplus P(n))$ if $n\geq 1$ and $\mathcal{B}_0=\mathcal{A}_1$. These algebras have the same quiver as $\mathcal{A}_{n}$ ($\mathcal{A}_1$ if $n=0$) and these correspond to the finite-type blocks of basic algebras of the group algebras of a symmetric group. Of course, $\mathcal{B}_0$ corresponds to the simple block. Hence, the interesting case lies in $n\geq 1$ which we will assume from now on. Moreover, $(\mathcal{A}_{n+1}, P(1)\oplus \cdots \oplus P(n))$ is a quasi-hereditary cover of $\mathcal{B}_n$ and $\mathcal{A}_{n+1}$ has global dimension and dominant dimension equal to $2n$. In particular, $\mathcal{A}_{n+1}\cong \End_{\mathcal{B}_n}(\mathcal{B}_n\oplus S_n)$, where $S_i$ denotes the simple module in the quiver of $\mathcal{B}_n$ corresponding to vertex $i$.  Thus, the Ringel dual of $\mathcal{A}_{n+1}$ is isomorphic to $\End_{\mathcal{B}_n}(\mathcal{B}_n \oplus \Omega^n(S_n))\cong \End_{\mathcal{B}_n}(\mathcal{B}_n\oplus S_1)$ by Theorem  \ref{Ringeldualtheorem}, since $\Omega^n(S_n) \cong S_1$.
Note that due to the symmetry in the relations, we also have $\mathcal{A}_{n+1}\cong \End_{\mathcal{B}_n}(\mathcal{B}_n\oplus S_1)$ and $\mathcal{A}_{n+1}$ is Ringel self-dual. Thus, in contrast to the previous class of examples, one can have Ringel self-duality even when the corresponding modules $U\oplus M$ and $U\oplus \Omega^d(M)$ are not isomorphic.

\section*{Acknowledgments} 
Rene Marczinzik is funded by the DFG with the project number 428999796. Tiago Cruz is funded by the \emph{Studienstiftung des Deutschen Volkes}.
The proof of part (1) of Theorem \ref{endofstandardisfrobenius} was first proven by Volodymyr Mazorchuk, who shared it with the second author via email. We followed his sketch of the proof and thank him for allowing us to use this result in the present article. We thank Steffen Koenig for helpful comments and suggestions.

We profited from the use of the GAP-package \cite{QPA}.

\bibliographystyle{alphaurl}
\bibliography{ref}

\begin{thebibliography}{AHLU00}

\bibitem[ADL98]{ADL}
I.~\'{A}goston, V.~Dlab, and E.~Luk\'{a}cs.
\newblock \href{https://mr.math.ca/article/stratified-algebras/}{Stratified
  algebras}.
\newblock {\em C. R. Math. Acad. Sci. Soc. R. Can.}, 20(1):22--28, 1998.

\bibitem[AHLU00]{ADLU}
I.~\'{A}goston, D.~Happel, E.~Luk\'{a}cs, and L.~Unger.
\newblock Standardly stratified algebras and tilting.
\newblock {\em J. Algebra}, 226(1):144--160, 2000.
\newblock \href {https://doi.org/10.1006/jabr.1999.8154}
  {\path{doi:10.1006/jabr.1999.8154}}.

\bibitem[APT92]{APT}
M.~Auslander, M.~I. Platzeck, and G.~Todorov.
\newblock Homological theory of idempotent ideals.
\newblock {\em Trans. Amer. Math. Soc.}, 332(2):667--692, 1992.
\newblock \href {https://doi.org/10.2307/2154190} {\path{doi:10.2307/2154190}}.

\bibitem[AR91]{AR}
M.~Auslander and I.~Reiten.
\newblock Applications of contravariantly finite subcategories.
\newblock {\em Adv. Math.}, 86(1):111--152, 1991.
\newblock \href {https://doi.org/10.1016/0001-8708(91)90037-8}
  {\path{doi:10.1016/0001-8708(91)90037-8}}.

\bibitem[ARS97]{ARS}
M.~Auslander, I.~Reiten, and S.~O. Smal\o.
\newblock {\em Representation theory of {A}rtin algebras}, volume~36 of {\em
  Cambridge Studies in Advanced Mathematics}.
\newblock Cambridge University Press, Cambridge, 1997.
\newblock Corrected reprint of the 1995 original.
\newblock \href {https://doi.org/10.1017/CBO9780511623608}
  {\path{doi:10.1017/CBO9780511623608}}.

\bibitem[AW92]{AW}
W.~A. Adkins and S.~H. Weintraub.
\newblock {\em Algebra}, volume 136 of {\em Graduate Texts in Mathematics}.
\newblock Springer-Verlag, New York, 1992.
\newblock An approach via module theory.
\newblock \href {https://doi.org/10.1007/978-1-4612-0923-2}
  {\path{doi:10.1007/978-1-4612-0923-2}}.

\bibitem[CD05]{CheDl}
X.~Chen and V.~Dlab.
\newblock Properly stratified endomorphism algebras.
\newblock {\em J. Algebra}, 283(1):63--79, 2005.
\newblock \href {https://doi.org/10.1016/j.jalgebra.2004.09.002}
  {\path{doi:10.1016/j.jalgebra.2004.09.002}}.

\bibitem[CE18]{CE}
T.~Conde and K.~Erdmann.
\newblock The {R}ingel dual of the {A}uslander-{D}lab-{R}ingel algebra.
\newblock {\em J. Algebra}, 504:506--535, 2018.
\newblock \href {https://doi.org/10.1016/j.jalgebra.2018.02.029}
  {\path{doi:10.1016/j.jalgebra.2018.02.029}}.

\bibitem[Che17]{Che}
X.~Chen.
\newblock \href{https://arxiv.org/abs/1712.04587}{Gorenstein Homological
  Algebra of Artin Algebras}, 2017.
\newblock \href {http://arxiv.org/abs/1712.04587} {\path{arXiv:1712.04587}}.

\bibitem[CK17]{CheKo}
Y.~Chen and S.~Koenig.
\newblock Recollements of self-injective algebras, and classification of
  self-injective diagram algebras.
\newblock {\em Math. Z.}, 287(3-4):1009--1027, 2017.
\newblock \href {https://doi.org/10.1007/s00209-017-1857-4}
  {\path{doi:10.1007/s00209-017-1857-4}}.

\bibitem[CPS96]{CPS}
E.~Cline, B.~Parshall, and L.~Scott.
\newblock Stratifying endomorphism algebras.
\newblock {\em Mem. Amer. Math. Soc.}, 124(591):viii+119, 1996.
\newblock \href {https://doi.org/10.1090/memo/0591}
  {\path{doi:10.1090/memo/0591}}.

\bibitem[DK94]{DK}
Y.~A. Drozd and V.~V. Kirichenko.
\newblock {\em Finite-dimensional algebras}.
\newblock Springer-Verlag, Berlin, 1994.
\newblock Translated from the 1980 Russian original and with an appendix by
  Vlastimil Dlab.
\newblock \href {https://doi.org/10.1007/978-3-642-76244-4}
  {\path{doi:10.1007/978-3-642-76244-4}}.

\bibitem[Don98]{D}
S.~Donkin.
\newblock {\em The {$q$}-{S}chur algebra}, volume 253 of {\em London
  Mathematical Society Lecture Note Series}.
\newblock Cambridge University Press, Cambridge, 1998.
\newblock \href {https://doi.org/10.1017/CBO9780511600708}
  {\path{doi:10.1017/CBO9780511600708}}.

\bibitem[FK11a]{FanKoe2}
M.~Fang and S.~Koenig.
\newblock Endomorphism algebras of generators over symmetric algebras.
\newblock {\em J. Algebra}, 332:428--433, 2011.
\newblock \href {https://doi.org/10.1016/j.jalgebra.2011.02.031}
  {\path{doi:10.1016/j.jalgebra.2011.02.031}}.

\bibitem[FK11b]{FanKoe3}
M.~Fang and S.~Koenig.
\newblock Schur functors and dominant dimension.
\newblock {\em Trans. Amer. Math. Soc.}, 363(3):1555--1576, 2011.
\newblock \href {https://doi.org/10.1090/S0002-9947-2010-05177-3}
  {\path{doi:10.1090/S0002-9947-2010-05177-3}}.

\bibitem[FK16]{FanKoe}
M.~Fang and S.~Koenig.
\newblock Gendo-symmetric algebras, canonical comultiplication, bar cocomplex
  and dominant dimension.
\newblock {\em Trans. Amer. Math. Soc.}, 368(7):5037--5055, 2016.
\newblock \href {https://doi.org/10.1090/tran/6504}
  {\path{doi:10.1090/tran/6504}}.

\bibitem[FM06]{FrMa}
A.~Frisk and V.~Mazorchuk.
\newblock Properly stratified algebras and tilting.
\newblock {\em Proc. London Math. Soc. (3)}, 92(1):29--61, 2006.
\newblock \href {https://doi.org/10.1017/S0024611505015431}
  {\path{doi:10.1017/S0024611505015431}}.

\bibitem[Gre07]{Gre}
J.~A. Green.
\newblock {\em Polynomial representations of {${\rm GL}_{n}$}}, volume 830 of
  {\em Lecture Notes in Mathematics}.
\newblock Springer, Berlin, augmented edition, 2007.
\newblock With an appendix on Schensted correspondence and Littelmann paths by
  K. Erdmann, Green and M. Schocker.
\newblock \href {https://doi.org/10.1007/3-540-46944-3}
  {\path{doi:10.1007/3-540-46944-3}}.

\bibitem[HU96]{HU}
D.~Happel and L.~Unger.
\newblock Modules of finite projective dimension and cocovers.
\newblock {\em Math. Ann.}, 306(3):445--457, 1996.
\newblock \href {https://doi.org/10.1007/BF01445260}
  {\path{doi:10.1007/BF01445260}}.

\bibitem[Hum08]{H}
J.~E. Humphreys.
\newblock {\em Representations of semisimple {L}ie algebras in the {BGG}
  category {$\mathcal{O}$}}, volume~94 of {\em Graduate Studies in
  Mathematics}.
\newblock American Mathematical Society, Providence, RI, 2008.
\newblock \href {https://doi.org/10.1090/gsm/094} {\path{doi:10.1090/gsm/094}}.

\bibitem[HX10]{HX}
W.~Hu and C.~Xi.
\newblock Derived equivalences and stable equivalences of {M}orita type, {I}.
\newblock {\em Nagoya Math. J.}, 200:107--152, 2010.
\newblock \href {https://doi.org/10.1215/00277630-2010-014}
  {\path{doi:10.1215/00277630-2010-014}}.

\bibitem[IS18]{IyaSol}
O.~Iyama and \O. Solberg.
\newblock Auslander-{G}orenstein algebras and precluster tilting.
\newblock {\em Adv. Math.}, 326:200--240, 2018.
\newblock \href {https://doi.org/10.1016/j.aim.2017.11.025}
  {\path{doi:10.1016/j.aim.2017.11.025}}.

\bibitem[Iya07]{Iya}
O.~Iyama.
\newblock Higher-dimensional {A}uslander-{R}eiten theory on maximal orthogonal
  subcategories.
\newblock {\em Adv. Math.}, 210(1):22--50, 2007.
\newblock \href {https://doi.org/10.1016/j.aim.2006.06.002}
  {\path{doi:10.1016/j.aim.2006.06.002}}.

\bibitem[Lak00]{Lak}
P.~Lakatos.
\newblock On a theorem of {V}. {D}lab.
\newblock {\em Algebr. Represent. Theory}, 3(1):99--103, 2000.
\newblock \href {https://doi.org/10.1023/A:1009922008930}
  {\path{doi:10.1023/A:1009922008930}}.

\bibitem[Mar18]{Mar}
R.~Marczinzik.
\newblock On minimal {A}uslander-{G}orenstein algebras and standardly
  stratified algebras.
\newblock {\em J. Pure Appl. Algebra}, 222(12):4068--4081, 2018.
\newblock \href {https://doi.org/10.1016/j.jpaa.2018.02.020}
  {\path{doi:10.1016/j.jpaa.2018.02.020}}.

\bibitem[Maz04a]{M}
V.~Mazorchuk.
\newblock
  \href{http://admjournal.luguniv.edu.ua/index.php/adm/article/view/1001}{On
  finitistic dimension of stratified algebras}.
\newblock {\em Algebra Discrete Math.}, 2004(3):77--88, 2004.

\bibitem[Maz04b]{M2}
V.~Mazorchuk.
\newblock Stratified algebras arising in {L}ie theory.
\newblock In {\em Representations of finite dimensional algebras and related
  topics in {L}ie theory and geometry}, volume~40 of {\em Fields Inst.
  Commun.}, pages 245--260. Amer. Math. Soc., Providence, RI, 2004.
\newblock \href {https://doi.org/10.1090/fic/040/12}
  {\path{doi:10.1090/fic/040/12}}.

\bibitem[MO04]{MazOv}
V.~Mazorchuk and S.~Ovsienko.
\newblock Finitistic dimension of properly stratified algebras.
\newblock {\em Adv. Math.}, 186(1):251--265, 2004.
\newblock \href {https://doi.org/10.1016/j.aim.2003.08.001}
  {\path{doi:10.1016/j.aim.2003.08.001}}.

\bibitem[MP04]{MP}
V.~Mazorchuk and A.~E. Parker.
\newblock On the relation between finitistic and good filtration dimensions.
\newblock {\em Comm. Algebra}, 32(5):1903--1916, 2004.
\newblock \href {https://doi.org/10.1081/AGB-120029912}
  {\path{doi:10.1081/AGB-120029912}}.

\bibitem[PR01]{PR}
M.~I. Platzeck and I.~Reiten.
\newblock Modules of finite projective dimension for standardly stratified
  algebras.
\newblock {\em Comm. Algebra}, 29(3):973--986, 2001.
\newblock \href {https://doi.org/10.1081/AGB-100001660}
  {\path{doi:10.1081/AGB-100001660}}.

\bibitem[QPA16]{QPA}
The QPA-team.
\newblock {\em QPA -- Quivers, path algebras and representations -- a GAP
  package, Version 1.25}, 2016.
\newblock URL: \url{https://folk.ntnu.no/oyvinso/QPA/}.

\bibitem[Rei07]{Rei}
I.~Reiten.
\newblock Tilting theory and homologically finite subcategories with
  applications to quasihereditary algebras.
\newblock In {\em Handbook of tilting theory}, volume 332 of {\em London Math.
  Soc. Lecture Note Ser.}, pages 179--214. Cambridge Univ. Press, Cambridge,
  2007.
\newblock \href {https://doi.org/10.1017/CBO9780511735134.008}
  {\path{doi:10.1017/CBO9780511735134.008}}.

\bibitem[SY11]{SkoYam}
A.~Skowro\'{n}ski and K.~Yamagata.
\newblock {\em Frobenius algebras. {I}}.
\newblock EMS Textbooks in Mathematics. European Mathematical Society (EMS),
  Z\"{u}rich, 2011.
\newblock Basic representation theory.
\newblock \href {https://doi.org/10.4171/102} {\path{doi:10.4171/102}}.

\bibitem[Xi02]{Xi}
C.~Xi.
\newblock Standardly stratified algebras and cellular algebras.
\newblock {\em Math. Proc. Cambridge Philos. Soc.}, 133(1):37--53, 2002.
\newblock \href {https://doi.org/10.1017/S0305004102005996}
  {\path{doi:10.1017/S0305004102005996}}.

\end{thebibliography}

\end{document}